\documentclass[journal,twoside]{IEEEtran}

\usepackage{cite}

\usepackage{tabu}
\usepackage{graphicx}
\usepackage{mathtools,amssymb,amsthm}
\usepackage{enumerate}
\usepackage{hyperref,xcolor}
\hypersetup{
colorlinks,
linkcolor={red!50!black},
citecolor={blue!50!black},
urlcolor={blue!80!black}
}

\theoremstyle{plain}
\newtheorem{theorem}{Theorem}[section]

\newtheorem{lemma}[theorem]{Lemma}
\newtheorem{corollary}[theorem]{Corollary}

\theoremstyle{definition}
\newtheorem{definition}[theorem]{Definition}
\newtheorem{example}[theorem]{Example}

\DeclareMathOperator{\diag}{diag}
\DeclareMathOperator{\tr}{tr}
\DeclareMathOperator{\M}{M}
\DeclareMathOperator{\W}{W}
\DeclareMathOperator{\D}{D}
\DeclareMathOperator{\eu}{\scriptscriptstyle E}
\DeclareMathOperator{\tv}{\scriptscriptstyle TV}
\DeclareMathOperator{\kl}{\scriptscriptstyle KL}
\DeclareMathOperator{\js}{\scriptscriptstyle JS}
\DeclareMathOperator{\hl}{\scriptscriptstyle H}
\let\O\undefined
\DeclareMathOperator{\O}{O}

\newcommand{\tp}{{\scriptscriptstyle\mathsf{T}}}
\newcommand{\p}{{\scriptscriptstyle+}}

\newcommand{\m}{{\scriptscriptstyle-}}

\newcommand{\PW}{\W^\m}
\newcommand{\EW}{\W^\p}
\newcommand{\Pm}{\Phi^\m}
\newcommand{\Em}{\Phi^\p}
\newcommand{\PD}{\D^\m}
\newcommand{\ED}{\D^\p}
\newcommand{\Pd}{d^\m}
\newcommand{\Ed}{d^\p}
\newcommand{\dd}{d}
\newcommand{\Mp}{\M^p}

\makeatletter
\newcommand{\leqnomode}{\tagsleft@true}
\newcommand{\reqnomode}{\tagsleft@false}
\makeatother

\begin{document}
\title{Distances between probability distributions of different dimensions}
\author{Yuhang~Cai and Lek-Heng~Lim
\thanks{Y.~Cai is with the Department of Statistics, University of Chicago, Chicago, IL 60637 USA e-mail: yuhangc@uchicago.edu.}
\thanks{L.-H.~Lim is with the Computational and Applied Mathematics Initiative, University of Chicago, Chicago, IL 60637 USA e-mail: lekheng@uchicago.edu.}
}

\markboth{IEEE Transactions on Information Theory}{Cai, Lim: Distances between probability distributions of different dimensions}
\maketitle

\begin{abstract}
Comparing probability distributions is an indispensable and ubiquitous task in machine learning and statistics. The most common way to compare a pair of Borel probability measures is to compute a metric between them, and by far the most widely used notions of metric are the Wasserstein metric and the total variation metric. The next most common way is to compute a divergence between them, and in this case almost every known divergences such as those of Kullback--Leibler,  Jensen--Shannon, R\'enyi, and many more, are special cases of the $f$-divergence. Nevertheless these metrics and divergences may only be computed, in fact, are only defined, when the pair of probability measures are on spaces of the same dimension. How would one quantify, say, a KL-divergence between the uniform distribution on the interval $[-1,1]$ and a Gaussian distribution on $\mathbb{R}^3$? We show that these common notions of metrics and divergences give rise to natural distances between Borel probability measures defined on spaces of different dimensions, e.g., one on $\mathbb{R}^m$ and another on $\mathbb{R}^n$ where $m, n$ are distinct, so as to give a meaningful answer to the previous question.
\end{abstract}


\section{Introduction}
\label{sec:intro}

\IEEEPARstart{M}{easuring} a distance, whether in the sense of a metric or a divergence, between two probability distributions is a fundamental endeavor in machine learning and statistics. We encounter it in clustering \cite{irpino2008dynamic}, density estimation \cite{sheather2004density}, generative adversarial networks \cite{arjovsky2017wasserstein}, image recognition \cite{bao2017cvae}, minimax lower bounds \cite{guntuboyina2011lower},  and just about any field that undertakes a statistical approach towards data. It is well-known that the space of Borel probability measures on a measurable space $\Omega \subseteq \mathbb{R}^n$ may be equipped with many different metrics and divergences, each good for its own purpose, but two of the most common families are the $p$-Wasserstein metric
\begin{align*}
\W_p(\mu ,\nu )&\coloneqq \biggl[ \inf_{\gamma \in \Gamma(\mu ,\nu )} \int_{\Omega \times \Omega} \|x-y\|_2^p \, d\gamma(x,y) \biggr]^{1/p}
\intertext{and the $f$-divergence}
\D_{f}(\mu \Vert \nu ) &\coloneqq \int_\Omega f\Bigl(\frac{d\mu }{d\nu }\Bigr)d\nu .
\end{align*}
For $p=1$ and $2$, the $p$-Wasserstein metric gives the Kantorovich metric (also called earth mover's metric) and L\'evy-Fr\'echet metric respectively. Likewise, for various choices of $f$, we obtain as special cases the Kullback--Liebler,
Jensen--Shannon, R\'enyi, Jeffreys, Chernoff, Pearson chi-squared, Hellinger squared, exponential, and alpha--beta divergences, as well as the total variation metric (see Table~\ref{tab:f}). Nevertheless, a $p$-Wasserstein metric cannot be expressed as an $f$-divergence.

All these distances are only defined when $\mu$ and $\nu$ are probability measures on a common measurable space $\Omega \subseteq \mathbb{R}^n$. This article provides an answer to the question:
\begin{quote}\label{question}
How can one define a distance between $\mu$, a probability measure on $\Omega_1 \subseteq \mathbb{R}^m$, and $\nu$, a probability measure on $\Omega_2 \subseteq \mathbb{R}^n$, where $m \ne n$?
\end{quote}
We will show that this problem has a natural solution that works for any of the aforementioned metrics and divergences in a way that is consistent with recent extensions of distances to inequidimensional covariance matrices \cite{LSY} and subspaces \cite{YL}. Although we will draw from the same high-level ideas in \cite{LSY,YL}, we require substantially different techniques in order to work with probability measures.

Given a $p$-Wasserstein metric or an $f$-divergence, which is defined between two probability measures of the same dimension, we show that it naturally defines \emph{two} different distances for probability measures $\mu$ and $\nu$ on spaces of different dimensions --- we call these the \emph{embedding distance} and \emph{projection distance} respectively. Both these distances are completely natural and are each befitting candidates for the distance we seek; the trouble is that there is not one but two of them, both equally reasonable. The punchline, as we shall prove, is that the two distances are always equal, giving us a unique distance defined on inequidimensional probability measures. We will state this result more precisely after introducing a few notations.

To the best of our knowledge --- and we have one of our referees to thank for filling us in on this --- the only alternative for defining a distance between probability measures of different dimensions is the \emph{Gromov--Wasserstein distance} proposed in \cite{Memoli}. As will be evident from our description below, we adopt a `bottom-up' approach that begins from first principles and requires nothing aside from the most basic definitions. On the other hand, the approach in \cite{Memoli} is a `top-down' one by adapting the vastly more general and powerful Gromov--Hausdorff distance to a special case. Our construction works with a wide variety of common metrics and divergences mentioned in first paragraph. Although the work in \cite{Memoli} is restricted to the $2$-Wasserstein metric, it is  conceivable that the framework therein would apply more generally to other metrics as well; however, it is not obvious how the framework might apply to divergences given that the Gromov--Hausdorff approach requires a metric. In the one case that allows a comparison, namely, applying the two different constructions to the $2$-Wasserstein metric to obtain distances on probability measures of different dimensions, they lead to different results. We are of the opinion that both approaches are useful although our simplistic approach is more likely to yield distances that have closed-form expressions or are readily computable, as we will see in Section~\ref{sec:eg}; the Gromov--Wasserstein distance tends to be NP-hard \cite{Quad} and closed-form expression are rare and not easy to obtain \cite{SDD}.

\subsection{Main result}

Let $\M(\Omega)$ denote the set of all Borel probability measures on  $\Omega \subseteq \mathbb{R}^n$  and let $\Mp(\Omega) \subseteq \M(\Omega)$ denote those with finite $p$th moments, $p \in \mathbb{N}$. For any $m,n \in \mathbb{N}$, $m \le n$, we write
\[
\O(m,n)\coloneqq\{V\in \mathbb{R}^{m\times n}: VV^\tp =I_m\},
\]
i.e., the Stiefel manifold of $m \times n$ matrices with orthonormal rows. We write $\O(n) \coloneqq \O(n,n)$ for the orthogonal group. For any $V \in \O(m,n)$ and $b \in \mathbb{R}^{m}$, let
\[
\varphi_{V,b} : \mathbb{R}^n \to \mathbb{R}^m, \qquad \varphi_{V,b}(x) = Vx+b;
\]
and for any $\mu \in \M(\mathbb{R}^n)$, let $\varphi_{V,b}(\mu )\coloneqq \mu \circ \varphi_{V,b}^{-1}$ be the pushforward measure. For simplicity, we write $\varphi_V \coloneqq \varphi_{V,0}$ when $b=0$. More generally, for any measurable map $\varphi: \mathbb{R}^n \to \mathbb{R}^m$, we let $\varphi(\mu)\coloneqq \mu \circ \varphi$ denote the pushforward measure.

For any $m,n \in \mathbb{N}$, there is no loss of generality in assuming that $m \le n$ for the remainder of our article. Our goal is to define a distance $d(\mu,\nu)$ for  measures $\mu \in  \M(\Omega_1) $ and $\nu \in  \M(\Omega_2) $ where $\Omega_1 \subseteq \mathbb{R}^m$ and $\Omega_2 \subseteq \mathbb{R}^n$, and where by `distance' we include both metrics and divergences. Again, there is no loss of generality in assuming that
\begin{equation}\label{eq:assume}
\Omega_1 = \mathbb{R}^m, \qquad \Omega_2 = \mathbb{R}^n
\end{equation}
since we may simply restrict our attention to measures supported on smaller subsets. Henceforth, we will assume \eqref{eq:assume}. We call $\mu$ and $\nu$ an $m$- and $n$-dimensional measure  respectively. 

We begin by defining the projection and embedding of measures. These are measure theoretic analogues of the Schubert varieties in \cite{YL} and we choose notations similar to \cite{YL}.
\begin{definition}\label{def:proemb}
Let $m,n \in \mathbb{N}$, $m \le n$. For any $\mu \in \M(\mathbb{R}^m)$ and $\nu \in \M(\mathbb{R}^n)$, 
the \emph{embeddings} of $\mu$ into $\mathbb{R}^n$ are the set of $n$-dimensional measures
\begin{multline*}
\Em(\mu,n)\coloneqq\{\alpha \in \M(\mathbb{R}^n):\varphi_{V,b}(\alpha )=\mu \\
\text{for some } V \in \O(m,n),\; b\in \mathbb{R}^m \};
\end{multline*}
and the \emph{projections} of $\nu $ onto $\mathbb{R}^m$ are the set of $m$-dimensional measures
\begin{multline*}
\Pm(\nu ,m) \coloneqq \{\beta\in \M(\mathbb{R}^m): \varphi_{V,b}(\nu )=\beta \\
\text{for some } V \in \O(m,n),\; b\in \mathbb{R}^m \}.
\end{multline*}
Let $d$ be any notion of distance on $\M(\mathbb{R}^n)$ for any $n \in \mathbb{N}$. Define the \emph{projection distance}
\[
d^\m(\mu,\nu )\coloneqq\inf_{\beta \in \Pm(\nu ,m)}d(\mu, \beta)
\]
and the \emph{embedding distance}
\[
d^\p(\mu,\nu )\coloneqq\inf_{\alpha \in \Em(\mu,n)} d(\alpha ,\nu ).
\]
\end{definition}
Both $d^\m(\mu,\nu )$ and $d^\p(\mu,\nu )$ are natural ways of defining $d$ on probability measures $\mu$ and $\nu$ of different dimensions. The trouble is that they are just as natural and there is no reason to favor one or the other. Our main result, which resolves this dilemma, may be stated as follows.
\begin{theorem}\label{thm:main}
Let $m,n \in \mathbb{N}$, $m \le n$. Let $d$ be a $p$-Wasserstein metric or an $f$-divergence.
Then
\begin{equation}\label{eq:equal}
d^\m(\mu,\nu ) = d^\p(\mu,\nu ).
\end{equation}
\end{theorem}
The common value in \eqref{eq:equal}, denoted $\widehat{d}(\mu,\nu )$, defines a distance between $\mu$ and $\nu$ and serves as our answer to the question on page~\pageref{question}. We will prove Theorem~\ref{thm:main} for $p$-Wasserstein metric (Theorem~\ref{thm:W}) and for $f$-divergence (Theorem~\ref{thm:f}).  Jensen--Shannon divergence (Theorem~\ref{thm:JS}) and total variation metric (Theorem~\ref{thm:tv}) require separate treatments since the definition of $p$-Wasserstein metric requires that $\mu$ and $\nu$ have finite $p$th moments and the definition of $f$-divergence requires that $\mu$ and $\nu$ have densities, assumptions that we do not need for Jensen--Shannon divergence and total variation metric. While the proofs of Theorems~\ref{thm:W}, \ref{thm:f}, \ref{thm:JS}, and \ref{thm:tv} follow a similar broad outline, the subtle details are different and depend on the specific distance involved. 

An important departure from the results in \cite{LSY,YL} is that in general $\widehat{d}(\mu,\nu ) \ne d(\mu,\nu )$ when $m = n$ although the Gromov--Wasserstein distance \cite{Memoli} mentioned earlier also lacks this property. To see this, we state a more general corollary.
\begin{corollary}\label{cor:n&scond}
Let $m,n \in \mathbb{N}$, $m \le n$. Let $d$ be a $p$-Wasserstein metric, a Jensen--Shannon divergence, a total variation metric, or an $f$-divergence. Then $\widehat{d}(\mu,\nu)  = d^\m(\mu,\nu) = d^\p(\mu,\nu) = 0$ if and only if $\varphi_{V,b}(\nu) = \mu$ for some $V\in \O(m,n)$ and $b\in \mathbb{R}^m$.
\end{corollary}
Corollary~\ref{cor:n&scond} gives a necessary and sufficient condition for $\widehat{d}(\mu,\nu)$ to be zero, saying that this happens if only if the two measures $\mu$ and $\nu$ are rotated and translated copies of each other, modulo embedding in a higher-dimensional ambient space when $m \ne n$. For any $d$ that is not  rotationally invariant, we will generally have $\widehat{d}(\mu,\nu ) \ne d(\mu,\nu )$ when $m = n$.

The discussion in the previous paragraph notwithstanding, the distance $\widehat{d}$ has several nice  features. Firstly, it preserves certain well-known relations satisfied by the original distance $d$. For example, we know that for $p \le q$, the $p$- and $q$-Wasserstein metrics satisfy $\W_p(\mu, \nu) \le \W_q(\mu, \nu)$ for measures $\mu,\nu$ of the same dimension; we will see in Corollary~\ref{cor:p<q} that
\[
\widehat{\W}_p(\mu, \nu) \le \widehat{\W}_q(\mu, \nu)
\]
for measures $\mu,\nu$ of \emph{different} dimensions. 
Secondly, as our construction applies consistently across a wide variety of distances, both metrics and divergences, relations between different types of distances can also be preserved. For example,  the total variation metric and KL-divergence satisfy Pinker's inequality $d_{\tv}(\mu, \nu )^2\le 1/2\D_{\kl}(\mu \Vert \nu )$ for measures $\mu,\nu$ of the same dimension; we will see in Corollary~\ref{cor:pink} that 
\[
\widehat{d}_{\tv}(\mu, \nu )^2\le \frac{1}{2} \widehat{\D}_{\kl}(\mu \Vert \nu )
\]
for measures $\mu,\nu$ of \emph{different} dimensions. As another example, the Hellinger squared divergence and total variation metric satisfy $\D_{\hl}(\mu,\nu)^2 \le 2 d_{\tv}(\mu, \nu) \le \sqrt{2} \D_{\hl}(\mu,\nu)$
for measures $\mu,\nu$ of the same dimension; we will see in Corollary~\ref{cor:htv} that
\[
\widehat{\D}_{\hl}(\mu,\nu)^2 \le 2\widehat{d}_{\tv}(\mu, \nu) \le \sqrt{2} \widehat{\D}_{\hl}(\mu,\nu)
\]
for measures $\mu,\nu$ of \emph{different} dimensions.

Another advantage of our construction is that for some common distributions, the distance $\widehat{d}$ obtained often has closed-form expression or is readily computable,\footnote{To the extent afforded by the original distance $d$ --- if $d$ has no closed-form expression or is NP-hard, we would not expect  $\widehat{d}$ to be any different.} as we will see in Section~\ref{sec:eg}. In particular, we will have an explicit answer for the rhetorical question in the abstract: What is the KL-divergence between the uniform distribution on $[-1,1]$ and a Gaussian distribution on $\mathbb{R}^3$?

\subsection{Background}

For easy reference, we remind the reader of two results.
\begin{theorem}[Hahn Decomposition]\label{thm:Hahn}
Let $\Omega$ be a  measurable space and $\mu$ be a signed measure on the $\sigma$-algebra $\Sigma(\Omega)$. Then there exist  $P$ and $N \in \Sigma(\Omega)$ such that
\begin{enumerate}[\upshape (i)]
\item $P\cup N = \Omega$, $P\cap N=\varnothing$;
\item any $E \in \Sigma(\Omega)$ with $E \subseteq P$ has $\mu(E)\ge 0$;
\item any $E \in \Sigma(\Omega)$ with $E \subseteq N$ has $\mu(E)\le 0$.
\end{enumerate}
\end{theorem}
The Disintegration Theorem \cite{Pachl} rigorously defines the notion of a nontrivial ``restriction'' of a measure to a measure-zero subset of a measure space. It is famously used to establish the existence of conditional probability measures.
\begin{theorem}[Disintegration Theorem]\label{thm:disin}
Let $\Omega_1$ and $\Omega_2$ be two Radon spaces. Let $\mu \in \M(\Omega_1)$ and $\varphi:\Omega_1 \to \Omega_2$ be a Borel measurable function. Set $\nu \in \M(\Omega_2)$ to be the pushforward measure $\nu =\mu \circ \varphi^{-1}$. Then there exists a $\nu $-almost everywhere uniquely determined family of probability measures $\{\mu _y \in \M(\Omega_1): y\in \Omega_2\} $ such that
\begin{enumerate}[\upshape (i)]
\item the function $\Omega_2 \to \M(\Omega_1)$, $y\mapsto \mu _y$ is Borel measurable, i.e., for any measurable $B \subseteq \Omega_1$, $y\to \mu _y(B)$ is a measurable function of $y$;
\item $\mu _y\bigl(\Omega_1 \setminus \varphi^{-1}(y)\bigr)=0$;
\item for every Borel-measurable function $f:\Omega_1 \to [0,\infty]$,
\[\int_{\Omega_1} f(x)\, d\mu (x)=\int_{\Omega_2} \int_{\varphi^{-1}(y)}f(x)\, d\mu _y(x)\,d\nu (y).
\]
\end{enumerate}
\end{theorem}
In this article, we use the terms `probability measure' and `probability distribution' interchangeably since given a cumulative distribution function $F$ and $A \in \Sigma(\Omega)$, $\mu(A) \coloneqq \int_{x\in A} dF(x)$ defines a probability measure.

\section{Wasserstein metric}\label{sec:wass}

We begin by properly defining the $p$-Wasserstein metric, filling in some details left out in Section~\ref{sec:intro}.  Given two measures $\mu , \nu \in \Mp(\mathbb{R}^n)$ and any $p \in [1,\infty]$, the \emph{$p$-Wasserstein metric}, also called the $L^p$-Wasserstein metric, between them is 
\begin{equation}\label{eq:wass}
\W_p(\mu ,\nu ) \coloneqq \biggl[ \inf_{\gamma \in \Gamma(\mu ,\nu )} \int_{\mathbb{R}^{2n}} \|x-y\|_2^p \, d\gamma(x,y) \biggr]^{1/p}
\end{equation}
where, as usual, $p = \infty$ is interpreted in the limiting sense of essential supremum. Here
\[
\Gamma(\mu ,\nu )\coloneqq \bigl\{\gamma\in \M(\mathbb{R}^{2n}):\pi_1^n(\gamma)= \nu, \; \pi_2^n(\gamma)=\mu \bigr\}
\]
is the set of \emph{couplings} between $\mu$ and $\nu $, where $\pi_1^n:\mathbb{R}^{2n}\to \mathbb{R}^n$ is the projection onto the first $n$ coordinates and $\pi_2^n:\mathbb{R}^{2n}\to \mathbb{R}^n$ the projection to the last $n$ coordinates. The measure $\pi \in \Gamma(\mu ,\nu )$ that attains the minimum in \eqref{eq:wass} is called the \emph{optimal transport coupling}.
For the purpose of this article, we use the standard Euclidean metric $d_{\eu}(x,y)=\|x - y\|_2$ but this may  be replaced by other metrics and $\mathbb{R}^n$ by other metric spaces; in which case \eqref{eq:wass} is just called the \emph{Wasserstein metric} or \emph{transportation distance}. The general definition is due to Villani \cite{villani2009optimal} but the notion has a long history involving the works Fr\'echet \cite{Frechet}, Kantorovich \cite{Kantorovich}, L\'evy \cite{Levy}, Wasserstein \cite{Wasserstein},  and many others. As we mentioned earlier, the $1$-Wasserstein metric is often called the earth mover's metric or Kantorovich metric whereas the 2-Wasserstein metric is sometimes called the L\'evy–Fr\'echet metric \cite{Frechet}.

The Wasserstein metric is widely used in the imaging sciences for capturing geometric features \cite{rubner2000earth,solomon2015convolutional,sandler2011nonnegative}, with a variety of applications including contrast equalization \cite{Delon2004Midway}, texture synthesis \cite{gutierrez2017optimal}, image matching \cite{Wang2013Optimal,Zhu2007Image}, image fusion \cite{courty2016optimal}, medical imaging \cite{wang2010optimal}, shape registration \cite{makihara2010earth}, image watermarking \cite{Mathon2014Optimal}. In economics, it is used to match job seekers with jobs,  determine  real estate prices, form matrimonial unions, among many other things \cite{Galichon2016OT}. Wasserstein metric and optimal transport coupling also show up unexpectedly in areas from astrophysics, where it is used to reconstruct initial conditions of the early universe \cite{frisch2002reconstruction}; to computer music, where it is used to automate music transcriptions \cite{flamary2016optimal}; to machine learning, where it is used for machine translation \cite{zhang2016building} and word embedding \cite{kusner2015word}.

Unlike the $f$-divergence  in Section~\ref{sec:f}, a significant advantage afforded by the Wasserstein distance is that it is finite even when neither measure is absolutely continuous with respect to the other. Our goal is to use $\W_p$ to construct a new distance $\widehat{\W}_p$  so that $\widehat{\W}_p(\mu ,\nu ) $ would be well-defined for  $\mu \in \M(\mathbb{R}^m)$ and $\nu \in \M(\mathbb{R}^n)$ where $m \ne n$. Note any attempt to directly extend the definition in \eqref{eq:wass} to such a scenario would require that we make sense of $\|x-y\|_2$ for $x \in \mathbb{R}^m$ and $y \in \mathbb{R}^n$ --- our approach would avoid this conundrum entirely.
We begin by establishing a simple but crucial lemma.
\begin{lemma}
\label{lem:projWasInequal}
Let $m,n \in \mathbb{N}$, $m \le n$, and $p \in [1,\infty]$. For any $\alpha , \nu \in \Mp(\mathbb{R}^n)$, any $ V \in \O(m,n)$, and any $b\in \mathbb{R}^m$, we have 
\[
\W_p\bigl(\varphi_{V,b}(\alpha ),\varphi_{V,b}(\nu )\bigr)\le \W_p(\alpha,\nu ).
\]
\end{lemma}
\begin{proof}
Let $\gamma \in \M(\mathbb{R}^{2n})$ be the optimal transport coupling for $\W_p(\alpha,v)$.  Consider the measurable map 
\[
    \overline{\varphi}_{V,b}: \mathbb{R}^{2n} \to \mathbb{R}^{2m}, \quad (x,y) \mapsto \bigl(\varphi_{V,b}(x),\varphi_{V,b}(y)\bigr),
\]
and define $\gamma_\p = \overline{\varphi}_{V,b}(\gamma)$. As $\varphi_{V,b} \circ \pi_i^n = \pi_i^m \circ \overline{\varphi}_{V,b}$,
\[
    \pi_1^m (\gamma_\p) = \varphi_{V,b}(\alpha),\quad \pi_2^m(\gamma_\p) = \varphi_{V,b}(\beta),    
\]
and thus $\gamma_\p (x,y)\in \Gamma\bigl(\varphi_{V,b}(\alpha),\varphi_{V,b}(\nu)\bigr)$. Now
\begin{align*}
\W_p(\varphi_{V,b}&(\alpha), \varphi_{V,b}(\nu)) \le
\int_{x\in \mathbb{R}^m,\; y\in \mathbb{R}^m}\|x-y\|_2^p \,d\gamma_\p (x,y)\\
&= \int_{z\in \mathbb{R}^n,\; w\in \mathbb{R}^n}\|\varphi_{V,b}(z)-\varphi_{V,b}(w)\|_2^p\,d\gamma(z,w)\\ 
&\le \int_{z\in \mathbb{R}^n,\; w\in \mathbb{R}^n}\|z-w\|_2^p\,d\gamma(z,w),
\end{align*}
and taking $p$th root gives the result. The last inequality follows from $\|\varphi_{V,b}(z) -\varphi_{V,b}(w)\|_2\le \|z-w\|_2$ as $\varphi_{V,b}$ is an orthogonal projection plus a translation.
\end{proof}
Lemma~\ref{lem:projWasInequal} assures that $\Pm(\mu,m) \subseteq \Mp(\mathbb{R}^m)$ but in general we may not have $\Em(\nu,n)\subseteq \Mp(\mathbb{R}^n)$. With this in mind, we introduce the set
\begin{multline*}
\Em_p(\mu,n) \coloneqq \{\alpha \in \Mp(\mathbb{R}^n): \varphi_{V,b}(\alpha) = \mu \\
\text{for some } V\in \O(m,n), \; b \in \mathbb{R}^{m}\}
\end{multline*}
for use in the next result, which shows that projection and embedding Wasserstein distances are always equal.
\begin{theorem}\label{thm:W}
Let $m,n \in \mathbb{N}$, $m \le n$, and $p \in [1,\infty]$. 
For $\mu\in \Mp(\mathbb{R}^m)$ and $\nu \in \Mp(\mathbb{R}^n)$, let
\begin{align*}
\PW_p(\mu,\nu )&\coloneqq\inf_{\beta \in \Pm(\nu ,m)}\W_p(\mu, \beta),\\
\EW_p(\mu,\nu )&\coloneqq\inf_{\alpha \in \Em_p(\mu,n)}\W_p(\alpha ,\nu ).
\end{align*}
Then
\begin{equation}\label{eq:W}
\PW_p(\mu,\nu )=\EW_p(\mu,\nu ).
\end{equation}
\end{theorem}
\begin{proof}
It is easy to deduce that $\PW_p(\mu, \nu )\le \EW_p(\mu,\nu )$: For any $\alpha \in \Em(\mu, n)$,  there exists $V_\alpha \in \O(m,n)$ and $b_\alpha \in \mathbb{R}^m$ with $\varphi_{V_\alpha, b_\alpha}(\alpha )=\mu$. It follows from Lemma~\ref{lem:projWasInequal} that $\W_p(\alpha ,\nu )\ge \W_p\bigl(\mu,\varphi_{V_\alpha, b_\alpha}(\nu )\bigr)$ and thus
\begin{align*}
\inf_{\alpha \in \Em(\mu,n)} \W_p(\alpha ,\nu )
&\ge\inf_{\alpha \in \Em(\mu, n)}\W_p\bigl(\mu, \varphi_{V_\alpha, b_\alpha}(\nu ) \bigr)\\
&\ge\inf_{V\in \O(m,n), \;b\in \mathbb{R}^m}\W_p\bigl(\mu, \varphi_{V,b}(\nu )\bigr).
\end{align*}

The bulk of the work is to show that $\PW_p(\mu,\nu )\ge \EW_p(\mu,\nu )$. Let $\varepsilon>0$ be arbitrary. Then there exists $\beta_* \in \Pm(\nu, m)$ with
\[
\W_p(\mu, \beta_*)\le \PW_p(\mu,\nu )+\varepsilon.
\]
Let $V_* \in \O(m,n)$ and $b_*\in \mathbb{R}^m$ be such that $\varphi_{V_*, b_*}(\nu) = \beta_*$ and $W_* \in \O(n-m,n)$ be such that $\begin{bsmallmatrix} V_*\\ W_*\end{bsmallmatrix} \in \O(n)$. Then $\varphi_{W_*}$ is the complementary projection of $\varphi_{V_*,b_*}$. Applying Theorem~\ref{thm:disin} to  $\varphi_{V_*, b_*}$, we obtain a family of measures $\{\nu_y \in  \M(\mathbb{R}^n) : y \in \mathbb{R}^m\}$ that satisfy
\[
\int_{\mathbb{R}^n} f(x)\,d\nu (x) =\int_{\mathbb{R}^m}\int_{\varphi_{V_*,b_*}^{-1}(y)} f(x)\, d\nu_y(x)\, d\beta_*(y)
\]
for any measurable function $f$.

Let $\gamma \in \Gamma(\beta_*,\mu)$ be the optimal transport coupling attaining $\W_p(\beta_*,\mu)$. Then
\[
\pi_1^m(\gamma) = \beta_*, \qquad
\pi_2^m(\gamma) = \mu.
\]
We define a new measure $\gamma_\p \in \M(\mathbb{R}^{2n})$ that will in turn give us a measure $\alpha_* \in \Em_p(\mu,n)$ with $\W_p(\alpha_*, \nu) \le \W_p(\mu, \beta_*)$. Firstly, we will define an intermediate probability measure $\tilde \gamma$ in $\M(\mathbb{R}^{n+m})$. For any measurable set $S\subseteq \mathbb{R}^{n+m}$, we define 
\[
    \tilde \gamma(S) \coloneqq \int_{(y,z)\in\mathbb{R}^{2m}}\int_{\varphi_{V_*,b_*}^{-1}(y)} \mathbb{I}_{(x, z)\in S}\,d\nu_y(x)\, d \gamma(y,z)
\]
where $\mathbb{I}$ denotes an indicator function. Consider the map
\[
    \rho:   \mathbb{R}^{n+m} \to \mathbb{R}^{2n}, \quad (x,z) \mapsto \bigl(x, V_*^\tp (z-b_*)+W_*^\tp W_*x\bigr)
\]
with $x\in \mathbb{R}^n$, $z \in \mathbb{R}^m$. Observe that this is an embedding of $\mathbb{R}^{n+m}$ into $\mathbb{R}^{2n}$. If we let $(x,y) = \rho\big((x,z)\big)$, then we find that $\varphi_{V_*,b_*}(y) = z$ and  $\varphi_{W_*}(y)= \varphi_{W_*}(x)$. We define $\gamma_\p$ to be the pushforward measure $\rho(\tilde \gamma)$. Next we will prove that $\pi_1^n(\gamma_\p) = \nu$. For any measurable set $S  \subseteq \mathbb{R}^n$, we have 
\begin{align*}
\pi_1^n&(\gamma_\p)(S) = \gamma_\p\big((\pi_1^n)^{-1}(S)\big) = \gamma_\p(S\times \mathbb{R}^n) \\
    &= \tilde \gamma\big(\rho^{-1}(S\times \mathbb{R}^n)\big) = \tilde \gamma (S\times \mathbb{R}^m)\\
    & =  \int_{(y,z)\in\mathbb{R}^{2m}}\int_{\varphi_{V_*,b_*}^{-1}(y)}\mathbb{I}_{(x, z)\in S\times \mathbb{R}^m}\,d\nu_y(x)\, d \gamma(y,z)\\ 
    & = \int_{(y,z)\in\mathbb{R}^{2m}}\int_{\varphi_{V_*,b_*}^{-1}(y)}\mathbb{I}_{x\in S}\,d\nu_y(x)\, d \gamma(y,z)\\
    & = \int_{(y,z)\in\mathbb{R}^{2m}}\int_{\varphi_{V_*,b_*}^{-1}(y)}\mathbb{I}_{x\in S}\,d\nu_y(x)\, d \beta_*(y) = \nu(S).
\end{align*}
Note that the first four equalities follow from the definition of the pushforward measure. For the next-to-last equality, observe that the indicator function $\mathbb{I}_{x\in S}$ is only a function of $y$. Hence $\nu (x)$ is a marginal measure of $\gamma_\p $.
Let $\alpha _*\in \M(\mathbb{R}^n)$ be defined by $\alpha _*=\pi_2^n(\gamma_\p)$. Then
\begingroup
\allowdisplaybreaks
\begin{align*}
&\int_{y\in \mathbb{R}^n} \|y\|_2^p \, d\alpha_*(y) = \int_{y\in \mathbb{R}^n} \bigl(\|y\|_2^{2}\bigr)^{p/2} \, d\alpha_*(y)\\
&= \int_{y\in \mathbb{R}^n} \bigl(\|\varphi_{V_*, b_*}(y)-b_*\|_2^{2} + \|\varphi_{W_*}(y)\|_2^{2}\bigr)^{p/2} \, d\alpha_*(y) \\
&\le \max \{2^{\frac{p-2}{2}}, 1\}  \int_{y\in \mathbb{R}^n}\hspace*{-3ex} \|\varphi_{V_*,b_*}(y)-b_*\|_2^{p} + \|\varphi_{W_*}(y)\|_2^{p} \, d\alpha_*(y)\\
& =\max \{2^{\frac{p-2}{2}}, 1\}  \biggl( \int_{y\in \mathbb{R}^m} \|y-b_*\|_2^p\, d \mu(y)  \\
&\qquad\qquad\qquad\qquad\qquad + \int_{(x,y)\in \mathbb{R}^{2n}} \|\varphi_{W_*}(y)\|_2^{p} \, d\gamma_\p(x,y)\biggr) \\
&\le \max \{2^{\frac{3p-4)}{2}}, 1\}  \biggl( \int_{y\in \mathbb{R}^m} \|y\|_2^p\, d \mu(y) +\|b_*\|_2^p \\
&\qquad\qquad\qquad\qquad\qquad+ \int_{x\in\mathbb{R}^n} \|\varphi_{W_*}(x)\|_2^p\, d \nu(x)\biggr) \\
&\le \max \{2^{\frac{3p-4}{2}}, 1\}  \biggl( \int_{y\in \mathbb{R}^m} \|y\|_2^p\, d \mu(y) + \int_{x\in\mathbb{R}^n} \|x\|_2^p\, d \nu(x)\biggr).
\end{align*}
\endgroup
Some explanation is in order: In the first inequality we have used $(\sigma +\tau)^{p/2} \le \max\{2^{\frac{p-2}{2}},1\}\cdot(\sigma^{p/2} + \tau^{p/2})$; in the fourth equality we observe that the support of $\gamma_\p$ is contained in the subspace $\{(x,y) : \varphi_{W_*}(x) = \varphi_{W_*}(y)\}$; in the fifth inequality we have used $\|\sigma-\tau\|_2 \le \|\sigma\|_2 + \|\tau\|_2$ and $(\sigma+\tau)^p \le 2^{p-1}(\sigma^p+ \tau^p)$; and in the last inequality, $\|\varphi_{W_*}(x)\|_2 \le \|x\|_2$. Since the $p$th central moment of $\alpha_*$ is bounded by the $p$th central moments of $\mu$ and $\nu$, we have $\alpha_* \in \Mp(\mathbb{R}^n)$. Finally,
\begin{align*}
\W_p^p&(\alpha_*,\nu )\le\int_{\mathbb{R}^{2n}} \|x-y\|_2^p\, d\gamma_\p (x,y)\\
& = \int_{\mathbb{R}^{2n}} \|\varphi_{V_*,b_*}(x)-\varphi_{V_*,b_*}(y)\|_2^p\, d\gamma_\p (x,y)\\
& = \int_{\mathbb{R}^{n+m}} \|\varphi_{V_*,b_*}(x)-z\|_2^p\, d\tilde \gamma (x,z)\\
&=\int_{(y,z)\in\mathbb{R}^{2m}}\int_{\varphi_{V_*,b_*}^{-1}(y)}\|\varphi_{V_*,b_*}(x)-z\|_2^p \,d\nu_y(x)\, d \gamma(y,z)\\
&=\int_{\mathbb{R}^{2m}}\|y-z\|_2^p\, d\gamma(y,z)=\W_p^p(\mu, \beta_*).
\end{align*}
Note that the first relation is an inequality as $\gamma_\p $ may not be an optimal transport coupling between $\alpha_*$ and $\nu$; the next equality follows from the support of $\gamma_\p$ being contained in the subspace $\{(x,y) : \varphi_{W_*}(x) = \varphi_{W_*}(y)\}$; and the next-to-last equality comes from the definition of pushforward measure.

We next show that $\varphi_{V_*, b_*}(\alpha _*)=\mu$ under the projection $\varphi_{V_*, b_*}$, i.e., $\alpha _*\in \Em_p(\mu,n)$. For a measurable $S\subseteq \mathbb{R}^m$, 
\begin{align*}
    \varphi_{V_*, b_*}&(\alpha _*)(S) = \alpha_*\big(\varphi_{V_*, b_*}^{-1}(S)\big)\\
   &= \gamma_\p \big( \mathbb{R}^n\times \varphi_{V_*, b_*}^{-1}(S)\big) = \tilde \gamma(\mathbb{R}^n\times S) \\
    & = \int_{(y,z)\in\mathbb{R}^{2m}}\int_{\varphi_{V_*,b_*}^{-1}(y)} \mathbb{I}_{(x,z)\in \mathbb{R}^n\times S }\,d\nu_y(x)\, d \gamma(y,z)\\
    & = \int_{(y,z)\in\mathbb{R}^{2m}}\int_{\varphi_{V_*,b_*}^{-1}(y)} \mathbb{I}_{z\in S }\,d\nu_y(x)\, d \gamma(y,z) = \mu(S),
\end{align*}
as required. Observe that the first three equalities are all consequences of the definition of a pushforward measure.
Therefore, with Lemma~\ref{lem:projWasInequal}, we have $\W_p(\alpha_*, \nu) =  W_p(\mu, \beta_*)$. Hence  
\begin{align*}
\EW_p(\mu,\nu ) &=\inf_{\alpha \in \Em(\mu,n)}\W_p(\alpha ,\nu )\le \W_p(\alpha _*,\nu )\\
&= \W_p(\mu,\beta_*)\le \PW_p(\mu,\nu ) +\varepsilon .
\end{align*}
Since $\varepsilon > 0$ is arbitrary, $\PW_p(\mu,\nu )\ge \EW_p(\mu,\nu )$.
\end{proof}
We denote the common value in \eqref{eq:W} by $\widehat{\W}_p(\mu, \nu)$, and call it the \emph{augmented $p$-Wasserstein distance} between $\mu \in \M_p(\mathbb{R}^m)$ and $\nu \in \M_p(\mathbb{R}^n)$. Note that this is a distance in the sense of a distance from a point to a set; it is not a metric since if we take $ \mu ,\nu \in \Mp(\mathbb{R}^m)$ with $\nu $ a nontrivial rotation of $\mu $, we will have $\widehat{\W}_p(\mu ,\nu ) = 0$ even though $\mu \ne \nu$.

The augmented $p$-Wasserstein distance $\widehat{\W}_p$ preserves some properties of the $p$-Wasserstein metric $\W_p$, an example is the following inequality, which is known to hold for $\W_p$.
\begin{corollary}\label{cor:p<q}
Let $m,n \in \mathbb{N}$, $m \le n$. Let $p,q \in [1,\infty]$, $p \le q$. For any $\mu\in \M_q(\mathbb{R}^m)$ and $\nu \in \M_q(\mathbb{R}^n)$, we have
\[
\widehat{\W}_p(\mu, \nu) \le \widehat{\W}_q(\mu, \nu).
\]
\end{corollary}
\begin{proof}
Follows from $\widehat{\W}_p(\mu, \nu) = \inf_{\beta \in \Pm(\nu, m)}\W_p(\mu, \beta) \le  \inf_{\beta \in \Pm(\nu, m)}\W_q(\mu, \beta) =\widehat{\W}_q(\mu, \nu)$.
\end{proof}

\section{$f$-Divergence}\label{sec:f}

The most useful notion of distance on probability densities is often not a metric. Divergences are in general asymmetric and do not satisfy the triangle inequality. The Kullback--Leibler divergence \cite{KL1,KL2} is probably the best known example, ubiquitous in information theory, machine learning, and statistics. It is used to characterize relative entropy in information systems \cite{Shannon1948Entropy}, to measure randomness in continuous time series \cite{von2009finding}, to quantify information gain in comparison of statistical models of inference \cite{Chaloner1995bayes}, among other things.

The KL-divergence is a special limiting case of a R\'enyi divergence \cite{Renyi1961divergence}, which is in turn a special case of a vast generalization called the $f$-divergence \cite{Csi}.
\begin{definition}\label{def:fdiv}
Let $\mu ,\nu \in \M(\Omega)$ and $\mu $ be absolutely continuous with respect to $\nu $. For any convex function $f:\mathbb{R}\to \mathbb{R}$ with $f(1)=0$, the \emph{$f$-divergence} of $\mu $ from $\nu$ is 
\[
\D_{f}(\mu \Vert \nu ) = \int_\Omega f\Bigl(\frac{d\mu }{d\nu }\Bigr)\, d\nu = \int_\Omega f\bigl(g(x)\bigr)\, d\nu (x),
\]
with $g$  the Radon--Nikodym derivative $d\mu (x)= g(x)\, d\nu (x)$.
\end{definition}
Aside from the R\'enyi divergence, the $f$-divergence includes just about every known  divergences as special cases. These include the Pearson chi-squared \cite{pearson1900x}, Hellinger squared \cite{Hellinger1909distance},  Chernoff \cite{Chernoff1952div}, Jeffreys \cite{Jeffreys1961div}, alpha--beta \cite{eguchi1985differential}, Jensen--Shannon \cite{JS,nielsen2010family}, and exponential \cite{Exponential} divergences, as well as the total variation metric. For easy reference, we provide a list in Table~\ref{tab:f}. 
Note that taking limit as $\theta \to 1$ in the R\'enyi divergence gives us the KL-divergence.
\begin{table*}[!t]
\centering
\caption{In the following, $\zeta = (1-\theta)\mu+\theta\nu$, $\eta = (1-\theta)\nu + \theta \mu$, and $\theta,\phi \in (0,1)$.}
\footnotesize
\tabulinesep=0.75ex
\begin{tabu}{@{}lll}
& $f(t)$ & $\D_f(\mu \Vert \nu)$ \\
Kullback--Liebler & $ t\log t$ & $\displaystyle\int_\Omega \log \Bigl(\dfrac{d\mu }{d\nu }\Bigr)d\mu$ \\
Exponential & $ t\log^2 t$ & $\displaystyle\int_\Omega \log^2\Bigl(\dfrac{d\mu}{d\nu}\Bigr) d\mu$ \\
Pearson & $ (t-1)^2$ & $\displaystyle \int_\Omega \Bigl(\dfrac{d\mu }{d\nu }-1 \Bigr)^2  d\nu$ \\             
Hellinger  & $ (\sqrt{t}-1)^2$ & $\displaystyle\int_\Omega \Bigl[ \Bigl(\dfrac{d\mu}{d\nu}\Bigr)^{1/2} - 1\Bigr]^2  d\nu$ \\
Jeffreys & $ (t-1)\log t$ & $\displaystyle\int_\Omega \Bigl(\dfrac{d\mu}{d\nu} - 1\Bigr)\log\Bigl(\dfrac{d\mu}{d\nu}\Bigr)  d\nu$ \\
R\'enyi & $ \dfrac{t^\theta -t}{\theta(\theta-1)}$ & $\dfrac{1}{\theta(\theta-1)}\displaystyle\int_\Omega \Bigl[\Bigl(\dfrac{d\mu}{d\nu}\Bigr)^\theta - \dfrac{d\mu}{d\nu} \Bigr] d\nu$ \\
Chernoff & $ \dfrac{4(1-t^{(1+\theta)/2})}{1-\theta^2}$ & $\dfrac{4}{1-\theta^2}\displaystyle\int_\Omega \Bigl[1-\Bigl(\dfrac{d\mu}{d\nu} \Bigr)^{(1+\theta)/2}\Bigr]  d\nu$ \\
alpha-beta & $\dfrac{2(1-t^{(1-\theta)/2})(1-t^{(1-\phi)/2})}{(1-\theta)(1-\phi)}$ & $\dfrac{2}{(1-\theta)(1-\phi)}\displaystyle\int_\Omega \Bigl[1-\Bigl( \dfrac{d\mu}{d\nu}\Bigr)^{(1-\theta)/2}\Bigr]\Bigl[1-\Bigl(\dfrac{d\mu}{d\nu}\Bigr)^{(1-\phi)/2}\Bigr]  d\nu$\\
Jensen--Shannon & $\dfrac{t}{2}\log \dfrac{t}{(1-\theta)t+\theta} + \dfrac{1}{2}\log \dfrac{1}{1-\theta+\theta t}$ & $\displaystyle \dfrac{1}{2}\int_\Omega \log \Bigl(\dfrac{d\mu }{d\zeta }\Bigr)d\mu + \dfrac{1}{2} \int_\Omega \log \Bigl(\dfrac{d\nu }{d\eta }\Bigr)d\nu$ \\
total variation & $ \dfrac{|t-1|}{2}$ & $\sup_{A\in \Sigma(\Omega)} |\mu (A)-\nu (A)|$
\end{tabu}
\label{tab:f}
\end{table*}

These divergences are all useful in their own right. The Pearson chi-squared divergence is used in statistical test of categorical data to quantify the difference between two distributions \cite{pearson1900x}. The Hellinger squared divergence is used in dimension reduction for multivariate data \cite{Rao1995Hellinger}. The Jeffreys divergence is used in Markov random field for image classification \cite{nishii2006image}. The Chernoff divergence is used in image feature classification, indexing, and retrieval \cite{hero2001alpha}. The R\'enyi divergence is used in quantum information theory as a measure of entanglement \cite{hastings2010measuring}. The alpha--beta divergence is used in geometrical analyses of parametric inference \cite{Amari1987Diff}. We will defer discussions of Jensen--Shannon divergence and total variation metric in Sections~\ref{sec:js} and \ref{sec:tv} respectively.

By definition, an $f$-divergence $\D_{f}(\mu \Vert \nu )$ is only defined if $\mu$ is absolutely continuous with respect to $\nu$. For convenience, in this section we will restrict our attention to probability measures with densities so that we do not have to keep track of which measure is absolutely continuous to which other measure.
Let $\lambda^n$ be the Lebesgue measure restricted to $\Omega \subseteq \mathbb{R}^n$. With respect to $\lambda^n$, we define
\begin{align*}
\M_d(\Omega) &\coloneqq \{\mu \in \M(\Omega) : \mu \text{ has density} \},\\
\M_{pd}(\Omega)&\coloneqq \{\mu \in \M_d(\Omega) : \mu \text{ has strictly positive density} \}.
\end{align*}
Note that $\mu \in \M_d(\Omega)$ iff it is absolutely continuous with respect to $\lambda^n$.
The following lemma guarantees the existence of  projection and embedding $f$-divergences, to be defined later. 
\begin{lemma}\label{lem:f-div1}
Let $m,n \in \mathbb{N}$, $m\le n$, and $\Pm(\nu ,m)$ be as in Definition~\ref{def:proemb}.
\begin{enumerate}[\upshape (i)]
\item\label{it:a} If $\nu \in \M_{pd}(\mathbb{R}^n)$, then $\Pm(\nu ,m)\subseteq \M_{pd}(\mathbb{R}^m)$.
\item\label{it:b} If $\alpha \in \M_d(\mathbb{R}^n)$ and $\nu \in \M_{pd}(\mathbb{R}^n)$, then $\alpha$ is absolutely continuous with respect to $\nu $.
\end{enumerate}
\end{lemma}
\begin{proof}
Let $\beta\in \Pm(\nu,m)$ and let $V \in \O(m,n)$, $b\in \mathbb{R}^m$ be such that $\varphi_{V,b}(\nu)=\beta$. Let $d\nu(x)=t(x)\, d\lambda^n(x)$  and $W \in \O(n-m,n)$ be such that $\begin{bsmallmatrix} V\\ W\end{bsmallmatrix} \in \O(n)$. For any measurable $g: \mathbb{R}^m \to \mathbb{R}$,
\begin{align*}
\int_{y\in \mathbb{R}^m}g(y)\,d\beta(y) &= \int_{x\in \mathbb{R}^n} g(\varphi_{V,b}(x))\,d\nu(x) \\
&= \int_{x\in \mathbb{R}^n} g(\varphi_{V,b}(x))t(x)\,d\lambda^n(x)\\
&= \int_{y\in \mathbb{R}^m} g(y)t'(y)\,d\lambda^m(y),
\end{align*}
where $t'(y) =\int_{\varphi_{V,b}^{-1}(y)}t(x)\,d\lambda^{n-m}\big(\varphi_W(x)\big)$. This is because $\varphi_{V,b}$ is an orthogonal projection plus a translation and for any measurable function $f$,
\begin{multline*}
\int_{x\in \mathbb{R}^n} f(x)\, d\lambda^n(x)\\
= \int_{y\in \mathbb{R}^m} \int_{\varphi_{V,b}^{-1}(y)} f(x)\,d\lambda^{n-m}(\varphi_W(x))\,d\lambda^m(y),
\end{multline*}
where the existence of $\varphi_W$ follows from Theorem~\ref{thm:disin}. Hence  $d\beta(y) = t'(y)\, d\lambda^m(y)$ and we have \eqref{it:a}.
For \eqref{it:b}, suppose $d \alpha(x)=t_\alpha(x)\,d\lambda^n(x)$ and $d\nu(x)=t_\nu(x)\,d\lambda^n(x)$ with $t_\nu(x)>0$, then
\[
d\alpha(x)=\frac{t_\alpha(x)}{t_\nu(x)}\,d\nu(x). \qedhere
\]
\end{proof}

We deduce an $f$-divergence analogue of Lemma~\ref{lem:projWasInequal}.
\begin{lemma}\label{lem:ProjFdivInequal}
Let $m,n \in \mathbb{N}$, $m\le n$, and $f:\mathbb{R} \to \mathbb{R}$ be convex with $f(1)=0$. Let $\alpha \in \M_d(\mathbb{R}^n)$, $\nu \in \M_{pd}(\mathbb{R}^n)$, $V \in \O(m,n)$, and $b\in \mathbb{R}^m$.
Then 
\[
\D_{f}(\alpha \Vert \nu ) \ge \D_{f}\bigl(\varphi_{V,b}(\alpha)\Vert \varphi_{V,b}(\nu )\bigr).
\]
\end{lemma}
\begin{proof}
Let $\mu=\varphi_{V,b}(\alpha)$ and $\beta=\varphi_{V,b}(\nu)$ with $d\alpha(x) = t(x)\, d\nu (x)$. By Theorem~\ref{thm:disin}, for any measurable function $g$,
\begin{align*}
\int_{\mathbb{R}^n} g(x)\,d\alpha(x) &= \int_{\mathbb{R}^m}\int_{\varphi_{V,b}^{-1}(y)}g(x)\,d\alpha_{y}(x) \,d\mu(y), \\
\int_{\mathbb{R}^n} g(x)\,d\nu(x) &= \int_{\mathbb{R}^m}\int_{\varphi_{V,b}^{-1}(y)}g(x)\,d \nu_{y}(x) \,d\beta(y).
\end{align*}
By Lemma~\ref{lem:f-div1}, we have $d\mu(y)= t'(y)\, d\beta(y)$ with $t'(y) =\int_{\varphi_{V,b}^{-1}(y)}t(x) \,d\nu_y(x)$. By Definition~\ref{def:fdiv} and  Jensen inequality,
\begin{align*}
\D_{f}(\alpha \Vert v)&= \int_{x\in \mathbb{R}^n} f\bigl(t(x)\bigr)\,d\nu(x) \\
&= \int_{y\in \mathbb{R}^m} \biggl[\int_{\varphi_{V,b}^{-1}(y)} f\bigl(t(x)\bigr)\,d\nu_y(x)\biggr] d\beta(y)\\ 
& \ge \int_{y\in \mathbb{R}^m}f\biggl[\int_{\varphi_{V,b}^{-1}(y)}t(x)\,d\nu_{y}(x)\biggr] d\beta(y) \\
&= \int_{y\in \mathbb{R}^m} f\bigl(t'(y)\bigr)\,d\beta(y) = \D_{f}(\mu\Vert \beta). \qedhere
\end{align*}
\end{proof}

Lemma~\ref{lem:f-div1} assures that $\Pm(\nu ,m)\subseteq \M_d(\mathbb{R}^m)$ but in general, it will not be true that $\Em(\mu, n) \subseteq \M_d(\mathbb{R}^n)$. As such we introduce the following subset:
\begin{multline*}
\Em_d(\mu, n) \coloneqq \{\alpha \in \M_d(\mathbb{R}^n): \varphi_{V,b}(\alpha )=\mu\\
 \text{for some } V \in \O(m,n),\; b\in \mathbb{R}^m\}
\end{multline*}
and with this, we establish Theorem~\ref{thm:main} for $f$-divergence.
\begin{theorem}\label{thm:f}
Let $m,n \in \mathbb{N}$ and $m \le n$. 
For $\mu\in \M(\mathbb{R}^m)$ and $\nu \in \M(\mathbb{R}^n)$, let
\begin{align*}
\PD_f(\mu\Vert\nu )&\coloneqq\inf_{\beta \in \Pm(\nu ,m)}\D_f(\mu\Vert \beta),\\
\ED_f(\mu\Vert\nu )&\coloneqq\inf_{\alpha \in \Em_d(\mu,n)}\D_f(\alpha \Vert\nu ).
\end{align*}
Then
\begin{equation}\label{eq:f}
\PD_f(\mu \Vert\nu ) = \ED_f(\mu \Vert\nu ).
\end{equation}
\end{theorem}
\begin{proof}
Again, $\PD_f(\mu\Vert \nu )\le \ED_f(\mu \Vert \nu )$ is easy: For any $\alpha \in \Em_d(\mu, n)$,  there exist $V_\alpha \in \O(m,n)$ and $b_\alpha \in \mathbb{R}^m$ with $\varphi_{V_\alpha,b_\alpha }(\alpha )=\mu$.
It follows from Lemma~\ref{lem:ProjFdivInequal} that $\D_f(\alpha \Vert \nu) \ge \D_f\bigl(\mu \Vert \varphi_{V_\alpha, b_\alpha}(\nu)\bigr)$ and thus
\begin{align*}
\inf_{\alpha \in \Em_d(\mu, n)}\D_{f}(\alpha \Vert \nu ) &\ge \inf_{\alpha \in \Em_d(\mu, n)}\D_{f}\bigl(\mu \Vert \varphi_{V_\alpha,b_\alpha }(\nu )\bigr)\\
&\ge \inf_{V \in \O(m,n),\;b\in \mathbb{R}^m}\D_{f}\bigl(\mu \Vert \varphi_{V,b}(\nu )\bigr).
\end{align*}
It remains to show $\PD_f(\mu\Vert\nu )\ge \ED_f(\mu\Vert\nu )$. By the definition of $\PD_f(\mu \Vert \nu)$, for any $\varepsilon > 0$, there exists $\beta_* \in \Pm(\nu, m)$ with
\[
\PD_{f}(\mu \Vert \nu )\le \D_{f}(\mu \Vert \beta_*)\le \PD_{f}(\mu \Vert \nu )+\varepsilon.
\]
Let $V_*\in \O(m,n)$ and $b_*\in \mathbb{R}^m$ be such that $\varphi_{V_*,b_*}(\nu) = \beta_*$ and $W_* \in \O(n-m,n)$ be such that $\begin{bsmallmatrix} V_*\\ W_*\end{bsmallmatrix} \in \O(n)$.
Applying Theorem~\ref{thm:disin} to $\varphi_{V_*,b_*}$, we obtain a family of measures $\{\nu _y \in \M(\mathbb{R}^n): y\in \mathbb{R}^m\}$ such that for any measurable function $f$,
\[
\int_{\mathbb{R}^n}f(x)\, d\nu (x) = \int_{\mathbb{R}^m} \int_{\varphi_{V_*,b_*}^{-1}(y)} f(x)\, d\nu_y(x)\, d \beta_*(y).
\]
Define $\alpha _* \in \M(\mathbb{R}^n)$ by
\[
\alpha_*(S) = \int_{\mathbb{R}^n}  \mathbb{I}_{x\in S} d\alpha _*(x) = \int_{\mathbb{R}^m} \int_{\varphi_{V_*, b_*}^{-1}(y)}\hspace*{-2ex} \mathbb{I}_{x\in S}\, d\nu_y(x)\, d \mu(y)
\]
for any measurable set $S\subseteq \mathbb{R}^n$.
Since $\nu \in \M_{pd}(\mathbb{R}^n)$, we may identify $\{\nu_y \in \M(\mathbb{R}^n): y\in \mathbb{R}^m\}$ as a subset of $\M_d(\mathbb{R}^{n-m})$. Let $d\nu_y(x) = s_y(x)\,d\lambda^{n-m}(\varphi_{W_*}(x))$ and  $d\mu(y)= g(y)\,d\lambda^m(y)$. Then 
\begin{align*}
d\alpha _*(x) &= g\bigl(\varphi_{V_*,b_*}(x)\bigr)s_{\varphi_{V_*,b_*}(x)}(x)\\
&\qquad\qquad d\lambda^{n-m}(\varphi_{W_*}(x)) d\lambda^{m}(\varphi_{V_*,b_*}(x))\\
&= g\bigl(\varphi_{V_*,b_*}(x)\bigr)s_{\varphi_{V_*,b_*}(x)}(x)\,d\lambda^n(x).
\end{align*}
Hence we deduce that $\alpha _*\in \M_d(\mathbb{R}^n)$. We may also check that $\varphi_{V_*,b_*}(\alpha _*)=\mu$. Let $d\mu(y)=t(y)\,d\beta_*(y)$. Then
\[
d\alpha_*(x) = t\bigl(\varphi_{V_*,b_*}(x)\bigr)\,d\nu (x).
\]
Finally, by Definition~\ref{def:fdiv}, we have
\begin{align*}
\ED_{f}(\mu\Vert \nu ) &\le \D_{f}(\alpha _*\Vert \nu )\\
&= \int_{\mathbb{R}^n} f\bigl(t(\varphi_{V_*,b_*}(x))\bigr)\,d\nu (x)\\
&= \int_{\mathbb{R}^m} f\bigl(t(y)\bigr)\,d\beta_*(y)\\
&= \D_{f}(\mu\Vert \beta_*) \le \PD_{f}(\mu\Vert \nu )+\varepsilon.
\end{align*}
Since $\varepsilon > 0$ is arbitrary, we have $\ED_{f}(\mu \Vert \nu )\le \PD_{f}(\mu \Vert \nu )$.
\end{proof}
As in the case of Wasserstein distance, we denote the common value in \eqref{eq:f} by $\widehat{\D}_f(\mu \Vert \nu)$ and call it the \emph{augmented} $f$-divergence, and likewise for all specific $f$-divergences.

Surprisingly, certain relations between these distances remain true with our extension to probability densities of different dimensions. For example, Pinker's inequality \cite{csiszar2011information} between the total variation metric $d_{\tv}$ and KL-divergence $\D_{\kl}$ holds for the augmented total variation distance $\widehat{d}_{\tv}$ and augmented KL-divergence $\widehat{\D}_{\kl}$; another standard relation between the Hellinger squared divergence $\D_{\hl}$ and total variation metric is preserved for their augmented counterparts too.
\begin{corollary}[Pinker's inequality for probability measures of different dimensions]\label{cor:pink}
Let $m,n \in \mathbb{N}$, $m \le n$. For any $\mu\in \M_d(\mathbb{R}^m)$ and $\nu \in \M_{pd}(\mathbb{R}^n)$, we have 
\[
\widehat{d}_{\tv}(\mu, \nu )\le \sqrt{\frac{1}{2} \widehat{\D}_{\kl}(\mu \Vert \nu )}.
\]
\end{corollary}
\begin{proof}
Follows from 
$\widehat{\D}_{\kl}(\mu\Vert \nu ) 
= \inf_{\beta\in \Pm(\nu ,m)}\D_{\kl}(\mu \Vert \beta)\ge \inf_{\beta\in \Pm(\nu ,m)}2 d_{\tv}(\mu \Vert \beta)^2 = 2 \widehat{d}_{\tv}(\alpha \Vert \nu )^2$.
\end{proof}

\begin{corollary}\label{cor:htv}
Let $m,n \in \mathbb{N}$, $m \le n$. For any $\mu\in \M_d(\mathbb{R}^m)$ and $\nu \in \M_{pd}(\mathbb{R}^n)$, we have 
\[
\widehat{\D}_{\hl}(\mu,\nu)^2 \le 2\widehat{d}_{\tv}(\mu, \nu) \le \sqrt{2} \widehat{\D}_{\hl}(\mu,\nu). 
\]
\end{corollary}
\begin{proof}
Clearly the two inequalities hold for $\D_{\hl}$ and $d_{\tv}$ when $m = n$. The inequidimensional version then follows from
\begin{align*}
\widehat{\D}_{\hl}(\mu,\nu)^2 &= \inf_{\beta\in \Pm(\nu ,m)} \D_{\hl}(\mu, \beta)^2 \\
&\le \inf_{\beta\in \Pm(\nu ,m)} 2d_{\tv}(\mu, \beta) = 2\widehat{d}_{\tv}(\mu, \nu)\\
&\le \inf_{\beta\in \Pm(\nu ,m)} \sqrt{2}\D_{\hl}(\mu, \beta)=\sqrt{2} \widehat{\D}_{\hl}(\mu,\nu).  \qedhere
\end{align*}
\end{proof}

\section{Jensen--Shannon divergence}\label{sec:js}

Let $\mu,\nu \in \M(\mathbb{R}^n)$ and $\theta \in (0,1)$. The \emph{Jensen--Shannon divergence} is  defined by
\begin{equation}\label{eq:js}
\D_{\js}(\mu,\nu ) \coloneqq \frac{1}{2}\D_{\kl}(\mu \Vert \zeta) + \frac{1}{2}\D_{\kl}(\nu \Vert \eta),
\end{equation}
where $\zeta \coloneqq (1-\theta)\mu+\theta\nu$ and $\eta \coloneqq  (1-\theta)\nu + \theta \mu$. What we call Jensen--Shannon divergence here is slightly more general \cite{nielsen2010family} than the usual definition \cite{JS},\footnote{Neither Jensen nor Shannon is a coauthor of \cite{JS}. The name comes from an application of Jensen inequality to Shannon entropy as a convex function to establish nonnegativity of the divergence.} which corresponds to the case when $\theta =1/2$.  When $\theta = 1$, we get the Jeffreys divergence in Table~\ref{tab:f} as another special case. We have written $\D_{\js}(\mu,\nu ) $ instead of the usual $\D_{\js} f(\mu \Vert \nu ) $ for $f$-divergence to remind the reader  that $\D_{\js}$ is symmetric in its arguments; in fact, $\D_{\js}(\mu,\nu )^{1/2} $ defines a metric on $\M(\mathbb{R}^n)$.

The Jensen--Shannon divergence is often viewed as the symmetrization of the Kullback--Liebler divergence but this perspective hides an important distinction, namely, the JS-divergence may be  defined on probability measures without  densities: Observe that $\mu,\nu$ are automatically absolutely continuous with respect to $\zeta$ and $\eta$. As such the definition in \eqref{eq:js} is valid for any $\mu,\nu \in \M(\mathbb{R}^n)$ and we do not need to work over $\M_d(\mathbb{R}^n)$ like in Section~\ref{sec:f}.

The JS-divergence is used in applications to compare genome  \cite{itzkovitz2010overlapping,sims2009alignment} and protein surfaces \cite{ofran2003analysing} in bioinformatics; to quantify information flow in social and biological systems \cite{dedeo2013bootstrap,  klingenstein2014civilizing}, and to detect anomalies in fire experiments \cite{mitroi2020parametric}. It was notably used to establish the main theorem in the landmark paper on Generative Adversarial Nets \cite{GAN}.

\begin{lemma}\label{lem:ProjJenShanInequal}
Let $m,n \in \mathbb{N}$, $m \le n$, and $\theta \in (0,1)$. For any $\alpha, \nu \in \M(\mathbb{R}^n)$, $V \in \O(m,n)$, and $b\in \mathbb{R}^m$,
\[
\D_{\js}(\alpha, \nu ) \ge \D_{\js}\bigl(\varphi_{V,b}(\alpha), \varphi_{V,b}(\nu )\bigr).
\]
\end{lemma}
\begin{proof}
The proof is similar to that of Lemma~\ref{lem:ProjFdivInequal}. We only need to check that $\varphi_{V,b}(\zeta) =(1-\theta)\varphi_{V,b}(\alpha)+ \theta\varphi_{V,b}(\nu)$, $\varphi_{V,b}(\eta) = \theta\varphi_{V,b}(\alpha)+ (1-\theta)\varphi_{V,b}(\nu)$, where $\zeta = (1-\theta)\mu+\theta\nu$, $\eta = (1-\theta)\nu + \theta \mu$.
\end{proof} 

We now prove Theorem~\ref{thm:main} for Jensen--Shannon divergence.
\begin{theorem}\label{thm:JS}
Let $m,n \in \mathbb{N}$ and $m \le n$. 
For $\mu\in \M(\mathbb{R}^m)$ and $\nu \in \M(\mathbb{R}^n)$, let
\begin{align*}
\PD_{\js}(\mu, \nu ) &\coloneqq\inf_{\beta \in \Pm(\nu ,m)}\D_{\js}(\mu, \beta),\\
\ED_{\js}(\mu, \nu ) &\coloneqq\inf_{\alpha \in \Em(\mu,n)}\D_{\js}(\alpha, \nu ).
\end{align*}
Then 
\begin{equation}\label{eq:JS}
\PD_{\js}(\mu, \nu ) = \ED_{\js}(\mu, \nu ).
\end{equation}
\end{theorem}
\begin{proof}
For any $\alpha \in \Em(\mu, n)$,  there exist $V_\alpha \in \O(m,n)$ and $b_\alpha \in \mathbb{R}^m$ with $\varphi_{V_\alpha, b_\alpha }(\alpha )=\mu$.
It follows from Lemma~\ref{lem:ProjJenShanInequal} that $\D_{\js}(\alpha, \nu) \ge \D_{\js}\bigl(\mu, \varphi_{V_\alpha,b_\alpha}(\nu)\bigr)$ and thus
\begin{align*}
\inf_{\alpha \in \Em(\mu, n)}\D_{\js}(\alpha, \nu ) &\ge \inf_{\alpha \in \Em(\mu, n)}\D_{\js}\bigl(\mu, \varphi_{V_\alpha, b_\alpha }(\nu )\bigr)\\
&\ge \inf_{V\in \O(m,n), \;b\in \mathbb{R}^m}\D_{\js}\bigl(\mu, \varphi_{V,b}(\nu )\bigr)
\end{align*}
and thus $\PD_{\js}(\mu,\nu )\le \ED_{\js}(\mu, \nu )$.

We next show that $\PD_{\js}(\mu, \nu )\ge \ED_{\js}(\mu, \nu )$. By the definition of $\PD_{\js}(\mu, \nu)$, for each $\alpha \in \Em(\mu, n)$ and any $\varepsilon > 0$, there exists $\beta_* \in \Pm(\nu, m)$ such that 
\[
\PD_{\js}(\mu ,\nu )\le \D_{\js}(\mu, \beta_*)\le \PD_{\js}(\mu,\nu )+\varepsilon.
\]
Let $V_*\in \O(m,n)$ and $b_* \in \mathbb{R}^m$ be such that $\varphi_{V_*,b_*}(\nu) = \beta_*$. Applying Theorem~\ref{thm:disin} to $\varphi_{V_*, b_*}$, we obtain $\{\nu _y \in \M(\mathbb{R}^n): y\in \mathbb{R}^m\}$ that satisfies
\[
\int_{\mathbb{R}^n} f(x)\, d\nu (x) = \int_{\mathbb{R}^m} \int_{\varphi_{V_*,b_*}^{-1}(y)}f(x)\, d\nu_y(x) \,d \beta_*(y)
\]
for any measurable function $f$. Let $\alpha _*\in \M(\mathbb{R}^n)$ be such that
\[
\alpha_*(S) = \int_{\mathbb{R}^n} \mathbb{I}_{x\in S}\, d\alpha _*(x) = \int_{\mathbb{R}^m} \int_{\varphi_{V_*,b_*}^{-1}(y)}\hspace*{-2ex} \mathbb{I}_{x\in S}\, d\nu_y(x) \,d \mu(y)
\]
for any measurable set $S\subseteq \mathbb{R}^n$. Then $\varphi_{V_*, b_*}(\alpha _*)=\mu$ and so $\alpha_*\in \Em(\mu, n)$. Consider the weighted measures 
\begin{align*}
\zeta^* &\coloneqq (1-\theta) \mu+\theta \beta_*, &\eta^* &\coloneqq \theta \mu+(1-\theta)\beta_*, \\
\xi_1 &\coloneqq (1-\theta ) \alpha _*+\theta \nu, &\xi_2 &\coloneqq \theta \alpha _*+(1-\theta)\nu.
\end{align*}
Since $\mu$ is absolutely continuous with respect to $\zeta^*$ and $\beta_*$  to $\eta^*$, we let $d\mu= g_1\,d\zeta^*$ and  $d\beta_* = g_2\,d\eta^*$. Then we have 
\begin{align*}
\varphi_{V_*, b_*}(\xi_1) &= \zeta^*, &d\alpha _*(x) &= g_1\bigl(\varphi_{V_*, b_*}(x)\bigr)\,d\xi_1(x), \\
\varphi_{V_*, b_*}(\xi_2) &= \eta^*,  &d\nu (x) &= g_2\bigl(\varphi_{V_*, b_*}(x)\bigr)\,d\xi_2(x).
\end{align*}
By the definition of $\ED_{\js}$,
\begin{align*}
\ED_{\js}(\mu, \nu ) &\le \D_{\js}(\alpha _*, \nu )= \frac{1}{2} \bigl[\D_{\kl}(\alpha _*\Vert \xi_1)+\D_{\kl}(\nu \Vert \xi_2)\bigr]\\
&= \frac{1}{2}\int_{\mathbb{R}^n} \log\bigl[g_1\bigl(\varphi_{V_*, b_*}(x)\bigr)\bigr]g_1\bigl(\varphi_{V_*, b_*}(x)\bigr)\, d\xi_1(x)\\
&\qquad\quad+\log\bigl[g_2\bigl(\varphi_{V_*, b_*}(x)\bigr)\bigr]g_2\bigl(\varphi_{V_*, b_*}(x)\bigr)\, d\xi_2(x)\\
&= \frac{1}{2}\int_{\mathbb{R}^m} \log\bigl(g_1(y)\bigr)g_1(y)\,d\zeta^*(y)\\
&\qquad\quad+\log\bigl(g_2(y)\bigr)g_2(y)\,d\eta^*(y) \\
&= \frac{1}{2}\bigl[\D_{\kl}(\mu \Vert \zeta^*)+\D_{\kl}(\beta_*\Vert \eta^*)\bigr]\\
&= \D_{\js}(\mu,\beta_*) \le \PD_{\js}(\mu,\nu )+\varepsilon.
\end{align*}
Since $\varepsilon >0$ is arbitrary, $\ED_{\js}(\mu, \nu )\le \PD_{\js}(\mu, \nu )$.
\end{proof}

\section{Total Variation Distance}\label{sec:tv}

The total variation metric is quite possibly the most classical notion of distance between probability measures. It is used in Markov models \cite{chen2014total,ding2010total}, stochastic processes \cite{nourdin2013convergence}, Monte Carlo algorithms \cite{brooks1997approach}, geometric approximation \cite{pekoz2013total}, image restoration \cite{osher2005iterative}, among other areas.

The definition is straightforward: The \emph{total variation metric} between $\mu ,\nu \in \M(\mathbb{R}^n)$ is simply
\[
d_{\tv}(\mu ,\nu ) \coloneqq \sup_{A\in \Sigma(\mathbb{R}^n)} |\mu (A)-\nu (A)|.
\]
As we saw in Section~\ref{sec:f}, when the probability measures have densities, the total variation metric is a special case of the $f$-divergence with $f(t) = \lvert t - 1 \rvert/2$. While it is not a special case of the Wasserstein metric in Section~\ref{sec:wass}, it is related in that
\[
d_{\tv}(\mu , \nu )=\inf _{\gamma\in \Gamma(\mu ,\nu )} \int_{\mathbb{R}^{2n}} \mathbb{I}_{x \neq y}(x,y) \, d\gamma(x,y),
\]
where $\Gamma(\mu ,\nu )$ is the set of couplings as in the definition of Wasserstein metric and $\mathbb{I}_{x \neq y}$ is an indicator function, i.e., takes value $1$ when $x\ne y$ and $0$ when $x = y$.

For any $\mu,\nu \in \M(\mathbb{R}^n)$, it follows from Hahn Decomposition, i.e., Theorem~\ref{thm:Hahn}, that there exists $S \in \Sigma(\mathbb{R}^n)$ with
\begin{equation}\label{eq:hahn0}
d_{\tv}(\mu ,\nu ) = \mu (S)-\nu (S).
\end{equation}
So for any measurable $B\subseteq S$, $ C\subseteq S^\complement  \coloneqq \mathbb{R}^n \setminus S$, 
\begin{equation}\label{eq:hahn1}
\mu (B)-\nu (B)\ge 0, \quad \mu (C)-\nu (C)\le 0;
\end{equation}
or, equivalently, for all measurable function $f(x)\ge 0$, 
\begin{equation}\label{eq:hahn2}
\begin{aligned}
\int_S f(x)\,d\bigl(\mu(x)-\nu (x)\bigr) &\ge 0 ,\\
\int_{S^\complement } f(x)\, d\bigl(\mu (x)-\nu (x)\bigr) &\le 0.
\end{aligned}
\end{equation}

\begin{lemma}\label{lem:tvprojinequal}
Let $\alpha ,\nu \in \M(\mathbb{R}^n)$. Then for any $V \in \O(m,n)$ and $b\in \mathbb{R}^m$,
\[
\dd_{\tv}(\alpha ,\nu ) \ge \dd_{\tv}\bigl(\varphi_{V,b}(\alpha ), \varphi_{V,b}(\nu )\bigr ).
\]
\end{lemma}
\begin{proof}
By \eqref{eq:hahn0} and \eqref{eq:hahn1},
$\dd_{\tv}\bigl(\varphi_{V,b}(\alpha),\varphi_{V,b}(\nu)\bigr) = \varphi_{V,b}(\alpha)(S)-\varphi_{V,b}(\nu)(S)
= \alpha\bigl(\varphi_{V,b}^{-1}(S)\bigr)-\nu\bigl(\varphi_{V,b}^{-1}(S)\bigr)\le \dd_{\tv}(\alpha,\nu)$.
\end{proof}
\begin{theorem}\label{thm:tv}
Let $m,n \in \mathbb{N}$ and $m \le n$. 
For $\mu\in \M(\mathbb{R}^m)$ and $\nu \in \M(\mathbb{R}^n)$, let
\begin{align*}
\Pd_{\tv}(\mu, \nu ) &\coloneqq\inf_{\beta \in \Pm(\nu ,m)}\dd_{\tv}(\mu, \beta),\\
\Ed_{\tv}(\mu, \nu ) &\coloneqq\inf_{\alpha \in \Em(\mu,n)}\dd_{\tv}(\alpha, \nu ).
\end{align*}
Then
\begin{equation}
\label{eq:tv}
\Ed_{\tv}(\mu,\nu ) = \Pd_{\tv}(\mu, \nu ).
\end{equation}
\end{theorem}
\begin{proof}
For any $\alpha \in \Em(\mu, n)$,  there exist $V_\alpha \in \O(m,n)$ and $b_\alpha \in \mathbb{R}^m$ with $\varphi_{V_\alpha,b_\alpha}(\alpha )=\mu$. So by Lemma~\ref{lem:tvprojinequal}, $\dd_{\tv}(\alpha, \nu) \ge \dd_{\tv}\bigl(\mu, \varphi_{V_\alpha, b_\alpha}(\nu)\bigr)$ and we get $\Ed_{\tv}(\mu,\nu ) \ge \Pd_{\tv}(\mu, \nu )$ from
\begin{align*}
\inf_{\alpha \in \Em(\mu, n)}\dd_{\tv}(\alpha, \nu ) &\ge \inf_{\alpha \in \Em(\mu, n)}\dd_{\tv}\bigl(\mu, \varphi_{V_\alpha,b_\alpha }(\nu )\bigr)\\
&\ge \inf_{V \in \O(m,n),b\in \mathbb{R}^m}\dd_{\tv}\bigl(\mu, \varphi_{V,b}(\nu )\bigr).
\end{align*}

We next show that $\Pd_{\tv}(\mu, \nu )\ge \Ed_{\tv}(\mu, \nu )$. By the definition of $\Pd_{\tv}(\mu, \nu)$, for any $\varepsilon > 0$, there exists $\beta_*\in \Pm(\nu,m)$ with
 \[\Pd_{\tv}(\mu,\nu ) \le d_{\tv}(\mu, \beta_*) \le \Pd_{\tv}(\mu,\nu )+\varepsilon.\]
Let $V_* \in \O(m,n)$ and $b_* \in \mathbb{R}^m$ be such that $\varphi_{V_*, b_*}(\nu) = \beta_*$ and $S\in \Sigma(\mathbb{R}^m)$ be such that $\dd_{\tv}(\mu, \beta_*) = \mu(S)-\beta_*(S)$. Applying Theorem~\ref{thm:disin} to $\varphi_{V_*, b_*}$, we obtain $\{\nu _y \in \M(\mathbb{R}^n): y\in \mathbb{R}^m\}$ that satisfies
\[
\int_{\mathbb{R}^n} f(x)\, d\nu (x) = \int_{\mathbb{R}^m} \int_{\varphi_{V_*,b_*}^{-1}(y)}f(x)\, d\nu_y(x) \,d \beta_*(y)
\]
for any measurable function $f$. Let $\alpha _*\in \M(\mathbb{R}^n)$ be such that
\[
\alpha_*(S) = \int_{\mathbb{R}^n} \mathbb{I}_{x\in S}\, d\alpha _*(x) = \int_{\mathbb{R}^m} \int_{\varphi_{V_*,b_*}^{-1}(y)}\hspace*{-2ex} \mathbb{I}_{x\in S}\, d\nu_y(x) \,d \mu(y)
\]
for any measurable set $S\subset \mathbb{R}^n$. We can check $\alpha_*$ is indeed a probability measure and $\varphi_{V_*, b_*}(\alpha_*)=\mu$. Partition $\mathbb{R}^n$ into $\varphi_{V_*, b_*}^{-1}(S)$ and $\varphi_{V_*, b_*}^{-1}(S^\complement )$. We claim that for any measurable $B\subseteq \varphi_{V_*, b_*}^{-1}(S)$ and $C\subseteq \varphi_{V_*, b_*}^{-1}(S^\complement )$,
\[
\alpha _*(B)-\nu (B)\ge 0, \quad \alpha _*(C)-\nu (C)\le 0.
\]
Let $g = \mathbb{I}_{x\in B}$. Then
\begin{align*}
\alpha _*(B)-\nu (B) &= \int_{\mathbb{R}^n} g(x)\,d\bigl(\alpha _*(x)-\nu (x)\bigr) \\
&= \int_{\mathbb{R}^m}\int_{\varphi_{V_*, b_*}^{-1}(y)} \hspace*{-2ex} g(x)\, d\nu_y(x) \,d \bigl(\mu(y)-\beta_*(y)\bigr)\\ 
& = \int_{S} \int_{\varphi_{V_*, b_*}^{-1}(y)} g(x) \,d\nu_y(x) \,d \bigl(\mu(y)-\beta_*(y)\bigr)\\
&= \int_{S} h(y)\,d\bigl(\mu(y)-\beta_*(y)\bigr), 
\end{align*}
where $h(y) = \int_{\varphi_{V_*, b_*}^{-1}(y)} g(x) d\nu_y(x) \ge 0. $  By \eqref{eq:hahn1}, we deduce that
$\alpha _*(B)-\nu (B) \ge 0$. Likewise, $\alpha _*(C)-\nu (C)\le 0$.
Let $T = \varphi_{V_*, b_*}^{-1}(S)$. Then for any measurable $A \subseteq \mathbb{R}^n$, 
\begin{align*}
\alpha_*(A)-\nu (A) &= \alpha _*(A\cap T)- \nu (A\cap T)\\
&\qquad\qquad+\alpha _*(A\cap T^\complement )-\nu (A\cap T^\complement )\\
&\le \alpha _*(A\cap T)- \nu (A\cap T)\le \alpha _*(T)- \nu (T).
\end{align*}
Hence we obtain
\begin{align*}
\Ed_{\tv}(\mu, \nu )&\le \dd_{\tv}(\alpha _*,\nu )= \alpha _*(T)-\nu(T) \\
&= \mu(S) - \beta_*(S)= d_{\tv}(\mu, \beta{_*}) \le \Pd_{\tv}(\mu, \nu )+\varepsilon.
\end{align*}
Since $\varepsilon > 0$ is arbitrary, $\Ed_{\tv}(\mu,\nu ) \le \Pd_{\tv}(\mu, \nu )$.
\end{proof}
Theorem~\ref{thm:tv} is stronger than what we may deduce from Theorem~\ref{thm:f} as measures do not required to have densities.

\section{Examples}\label{sec:eg}

Theorems~\ref{thm:W}, \ref{thm:f}, \ref{thm:JS}, and \ref{thm:tv} show that to compute any of the distances therein between probability measures of different dimensions, we  may either compute the projection distance $\Pd$ or the embedding distance $\Ed$. We will present five examples, three continuous and two discrete. In this section, we denote our probability measures by $\rho_1,\rho_2$ instead of $\mu, \nu$ to avoid any clash with the standard notation for mean.

In the following, we will write $\mathcal{N}_n(\mu, \Sigma)$ for the $n$-dimensional normal measure with mean $\mu \in \mathbb{R}^n$ and covariance $\Sigma \in \mathbb{R}^{n \times n}$. For $\rho_1 = \mathcal{N}_n(\mu_1, \Sigma_1)$ and $\rho_2 = \mathcal{N}_n(\mu_2, \Sigma_2) \in \M(\mathbb{R}^n)$,  recall that the $2$-Wasserstein metric and the KL-divergence between them are given by
\begin{align*}
\W_2^2(\rho_1, \rho_2) &= \| \mu_1 -\mu_2\|_2^2 + \tr\bigl(\Sigma_1 +\Sigma_2 - 2\Sigma_2^{\frac12}\Sigma_1 \Sigma_2^{\frac12}\bigr)^{\frac12},\\
\D_{\kl}(\rho_1 \Vert \rho_2) &= \frac{1}{2} \biggl[\tr(\Sigma_{2}^{-1} \Sigma_{1})+ (\mu_{2}-\mu_{1})^\tp \Sigma_{2}^{-1} (\mu_{2}-\mu_{1})\\
&\qquad\qquad -n+\log \biggl(\frac{\det \Sigma_{2}}{\det \Sigma_{1}}\biggr)\biggr]
\end{align*}
respectively. The former may be found in \cite{olkin1982distance} while the latter is a routine calculation.

We adopt the standard convention that a vector in $\mathbb{R}^m$ will always be assumed to be a column vector, i.e., $\mathbb{R}^m \equiv \mathbb{R}^{m \times 1}$. A matrix $X \in \mathbb{R}^{m \times n}$, when denoted $X = [x_1,\dots,x_n]$ implicitly means that $x_1,\dots,x_n \in \mathbb{R}^m$ are its column vectors, and when denoted $X = [y_1^\tp,\dots,y_m^\tp]^\tp$ implicitly means that $y_1,\dots,y_m \in \mathbb{R}^n$ are its row vectors. The notation $\diag(\lambda_1,\dots,\lambda_n)$ means an $n \times n$ diagonal matrix with diagonal entries $\lambda_1,\dots,\lambda_n \in \mathbb{R}$.

\begin{example}[$2$-Wasserstein distance between one- and $n$-dimensional Gaussians]\label{eg:W}
Let $\rho_1 = \mathcal{N}_1(\mu_1,\sigma^2) \in \M(\mathbb{R})$ be a one-dimensional Gaussian measure and $\rho_2 = \mathcal{N}_n(\mu_2,\Sigma) \in \M(\mathbb{R}^n)$ be an $n$-dimensional Gaussian measure, $n \in \mathbb{N}$ arbitrary.  We seek the $2$-Wasserstein distance $\widehat{W}_2(\rho_1, \rho_2 )$  between them. By Theorem~\ref{thm:W}, we have the option of computing either $\PW_2(\rho_1, \rho_2 ) $ or $\EW_2(\rho_1, \rho_2 ) $ but the choice is obvious, given that the former is considerably simpler:
\begin{align*}
\PW_2(\rho_1, \rho_2 )^2 &= \min_{\|x\|_2 = 1,\; y\in \mathbb{R}} \|\mu_1-x^\tp \mu_2-y\|_2^2\\
&\qquad\qquad+\tr(\sigma^2+x^\tp \Sigma x - 2\sigma \sqrt{x^\tp \Sigma x})\\
&=\min_{\|x\|_2 = 1} (\sigma - \sqrt{x^\tp \Sigma x})^2.
\end{align*}
Let $\lambda_1$ and $\lambda_n$ be the largest and smallest eigenvalues of $\Sigma$. Then $\lambda_n \le x^\tp \Sigma x \le \lambda_1$ and thus we must have
\[
\widehat{W}_2(\rho_1, \rho_2 ) = 
\begin{cases}
\sqrt{\lambda_n}-\sigma  &\text{if } \sigma < \sqrt{\lambda_n},\\
0 &\text{if } \sqrt{\lambda_n} \le \sigma  \le \sqrt{\lambda_1},\\
\sigma - \sqrt{\lambda_1} &\text{if } \sigma > \sqrt{\lambda_1}.
\end{cases}
\]
\end{example}

\begin{example}[KL-divergence between one- and $n$-dimensional Gaussians]
Let $\rho_1 = \mathcal{N}_1(\mu_1,\sigma^2)\in \M(\mathbb{R})$ and $\rho_2 = \mathcal{N}_n(\mu_2,\Sigma) \in \M(\mathbb{R}^n)$ be as in Example~\ref{eg:W}. By Theorem~\ref{thm:f}, we may compute either $\PD_{\kl}(\rho_1 \Vert \rho_2 ) $ or $\ED_{\kl}(\rho_1 \Vert \rho_2 )$ and again the simpler option is
\begin{align*}
\PD_{\kl}(\rho_1\Vert \rho_2 ) &= \min_{\|x\|_2=1,y\in \mathbb{R}}\frac{1}{2}\Bigl[\frac{\sigma^2}{x^\tp \Sigma x} + \frac{(\mu_1-x^\tp \mu_2-y)^2}{x^\tp \Sigma x} \\
&\qquad\qquad\qquad\qquad -1 + \log\Bigl(\frac{x^\tp \Sigma x}{\sigma^2}\Bigr)\Bigr] \\
&= \min_{\|x\|_2=1}\frac{1}{2}\Bigl[\frac{\sigma^2}{x^\tp \Sigma x} -1 + \log\Bigl(\frac{x^\tp \Sigma x}{\sigma^2}\Bigr)\Bigr]
\end{align*}
Again $\lambda_n \le x^\tp \Sigma x \le \lambda_1$ where $\lambda_1$ and $\lambda_n$ are the largest and smallest eigenvalues of $\Sigma$.  Since $f(\lambda) = \sigma^2/\lambda+ \log(\lambda/\sigma^2)$ has $f'(\lambda) = (\lambda-\sigma^2)/\lambda^2$, we obtain
\begin{multline*}
\widehat{\D}_{\kl}(\rho_1\Vert \rho_2 ) \\
=
\begin{cases}
\dfrac{1}{2} \Bigl[\dfrac{\sigma^2}{\lambda_n}-1 + \log\Bigl(\dfrac{\lambda_n}{\sigma^2}\Bigr)\Bigr] &\text{if } \sigma < \sqrt{\lambda_n},\\[2ex]
0 &\text{if } \sqrt{\lambda_n} \le \sigma  \le \sqrt{\lambda_1},\\[0.8ex]
\dfrac{1}{2} \Bigl[\dfrac{\sigma^2}{\lambda_1}-1 + \log\Bigl(\dfrac{\lambda_1}{\sigma^2}\Bigr)\Bigr] &\text{if } \sigma > \sqrt{\lambda_1}.
\end{cases}
\end{multline*}
\end{example}

\begin{example}[KL-divergence between uniform measure on $m$-dimensional ball and $n$-dimensional Gaussian]
Let $\mathbb{B}^m  =\{x\in \mathbb{R}^m : \|x\|_2\le 1\}$ be the unit $2$-norm ball in $\mathbb{R}^m$ and let $\rho_1 = \mathcal{U}(\mathbb{B}^m)$ be the uniform probability measure on $\mathbb{B}^m$. Let $\rho_2 = \mathcal{N}_n(\mu_2, \Sigma)$ be an $n$-dimensional Gaussian measure with zero mean and $\Sigma \in \mathbb{R}^{n\times n}$ symmetric positive definite. By Theorem~\ref{thm:f},
\begin{align*}
\widehat{\D}_{\kl}(\rho_1 \Vert \rho_2) &= \PD_{\kl}(\rho_1 \Vert \rho_2) \\
&= \inf_{V\in \O(m,n), \;b\in \mathbb{R}^m}\D_{\kl}\bigl(\rho_1 \Vert \varphi_{V,b}(\rho_2)\bigr). 
\end{align*}
Note that $\varphi_{V,b}(\rho_2) = \mathcal{N}_m(V\mu_2+b, V\Sigma V^\tp )$ is an $m$-dimensional Gaussian. Let $\lambda_1 \ge \dots \ge \lambda_n > 0$ be the eigenvalues of $\Sigma$ and  $\sigma_1 \ge \dots \ge \sigma_m > 0$ be the eigenvalues of $V\Sigma V^\tp $. Then 
\begin{align}
\D_{\kl}&\bigl(\rho_1 \Vert \varphi_{V,b}(\rho_2)\bigr) \nonumber\\
& = \frac12 \biggl[ \sum_{i=1}^m \log(\sigma_i) + \frac{1}{(m+2)}\sum_{i=1}^m\frac{1}{\sigma_i} \biggr] \nonumber\\
&\qquad\qquad +  \log \Gamma\Bigl(\frac{m}{2}+1\Bigr) + \frac{m\log 2}{2} \nonumber \\
&\qquad\qquad\qquad + \min_{b\in \mathbb{R}^m}(V\mu_2+b)^\tp V\Sigma V^\tp (V\mu_2+b) \nonumber \\
&=\frac12 \biggl[ \sum_{i=1}^m \log(\sigma_i) + \frac{1}{(m+2)}\sum_{i=1}^m\frac{1}{\sigma_i} \biggr] \nonumber\\
&\qquad\qquad +  \log \Gamma\Bigl(\frac{m}{2}+1\Bigr) + \frac{m\log 2}{2},  \label{eq:ug}
\end{align}
where the minimum is attained at $b=-V\mu_2$ and $\Gamma$ is the Gamma function.
Let $g(\sigma) \coloneqq \log(\sigma)/2 + 1/[2(m+2)\sigma]$, which has global minimum at $\sigma= 1/(m+2)$. For any $\alpha \ge \beta \ge 0$,
\begin{equation}\label{eq:gm}
g_m(\alpha, \beta) \coloneqq 
\begin{cases}
g(\beta) &\text{if } \beta > \frac{1}{m+2},\\
g\bigl(\frac{1}{m+2}\bigr) &\text{if } \beta \le \frac{1}{m+2} \le \alpha,\\
g(\alpha) &\text{if } \alpha < \frac{1}{m+2}.
\end{cases}
\end{equation}
Thus when $m=1$, we have
\begin{multline*}
\widehat{\D}_{\kl}(\rho_1 \Vert\rho_2) = g_1(\lambda_1, \lambda_n) + \dfrac12\log \dfrac{\pi}{2}\\
= 
\begin{cases}
\dfrac12\log \dfrac{\pi}{2} + \dfrac{1}{6\lambda_n}+ \dfrac{1}{2}\log \lambda_n &\text{if } \lambda_n > \frac13,\\[2ex]
\dfrac12\log \dfrac{\pi}{6} + \dfrac{1}{2} &\text{if } {\lambda_n} \le \frac13  \le {\lambda_1},\\[2ex]
\dfrac12\log \dfrac{\pi}{2} + \dfrac{1}{6\lambda_1}+ \dfrac{1}{2}\log \lambda_1 &\text{if } \lambda_1 < \frac13.
\end{cases}
\end{multline*}
Note that setting $n = 3$ answers the question we posed in the abstract: What is the KL-divergence between the uniform distribution $\rho_1 = \mathcal{U}([-1,1])$ and the Gaussian distribution $\rho_2 = \mathcal{N}_3(\mu_2, \Sigma)$ in $\mathbb{R}^3$.

More generally, suppose $m<n/2$. For any $\sigma_1\ge \dots\ge \sigma_m\ge 0$ with
\[
\lambda_{n-m+i}\le\sigma_i\le \lambda_i \qquad i=1,\dots,m,
\]
we construct $V\in \O(m,n)$ with $V\Sigma V^\tp= \diag(\sigma_1,\dots,\sigma_m)$. Let $\Sigma = Q\Lambda Q^\tp$ be an eigenvalue decomposition  with $Q = [q_1,\dots,q_n] \in \O(n)$.
For each $i=1,\dots,m$, let
\begin{align*}
v_i(\theta_i) &\coloneqq q_i\sin \theta_i + q_{n-m+i}\cos \theta_i \in \mathbb{R}^n, \\
\sigma_i(\theta_i) &\coloneqq \lambda_i\sin^2 \theta_i + \lambda_{n-m+i}\cos^2 \theta_i \in \mathbb{R}_\p.
\end{align*}
Then $V=[v_1(\theta_1)^\tp,\dots,v_m(\theta_m)^\tp]^\tp \in \O(m,n)$ and $V\Sigma V^\tp = \diag\bigl(\sigma_1(\theta_1),\dots,\sigma_m(\theta_m)\bigr) $. Choosing $\theta_i$ so that $\sigma_i(\theta_i) = \sigma_i$, $i =1,\dots,m$, gives us the required result.

With this observation, it follows that when $m < n/2$, the minimum in \eqref{eq:ug} is attained when $\sigma_i = \lambda_i$, $i =1,\dots,m$, and we obtain the closed-form expression
\[
\widehat{\D}_{\kl}(\rho_1 \Vert\rho_2) = \sum_{i=1}^m g_m(\lambda_i, \lambda_{n-m+i})+  \log \Gamma\Bigl(\frac{m}{2}+1\Bigr) + \frac{m\log 2}{2},
\]
where $g_m$ is as defined in \eqref{eq:gm}. 
\end{example}

\begin{example}[$2$-Wasserstein distance between Dirac measure on $\mathbb{R}^m$ and discrete measure on $\mathbb{R}^n$]\label{eg:dirac}
Let $y \in \mathbb{R}^m$ and $\rho_1 \in \M(\mathbb{R}^m)$ be the Dirac measure with $\rho_1(y)=1$,  i.e., all mass centered at $y$. Let $x_1,\dots,x_k \in \mathbb{R}^n$ be distinct points, $p_1,\dots,p_k \ge 0$, $p_1 + \dots + p_k = 0$, and let $\rho_2 \in \M(\mathbb{R}^n)$ be the discrete measure of point masses with $\rho_2 (x_i)=p_i$, $i=1,\dots,k$. We seek the $2$-Wasserstein distance $\widehat{W}_2(\rho_1, \rho_2 )$ and by Theorem~\ref{thm:W}, this is given by $\PW_2(\rho_1, \rho_2 )$. We will show it has a closed-form solution.
Suppose $m \le n$, then
\begin{align*}
\PW_2(\rho_1, \rho_2 )^2 &= \inf_{V\in \O(m,n),\; b\in \mathbb{R}^m}\sum_{i=1}^k p_i\|Vx_i+b+y\|^2_2 \\
&= \inf_{V\in \O(m,n)}\sum_{i=1}^k p_i \biggl\|Vx_i-\sum_{i=1}^kp_iV x_i \biggr\|^2_2\\
&= \inf_{V\in \O(m,n)} \tr(VX V^\tp),
\end{align*}
noting that the second infimum is attained by $b=-y-\sum_{i=1}^kp_iVx_i$ and defining $X$ in the last infimum to be
\[
X \coloneqq \sum_{i=1}^k p_i \biggl(x_i- \sum_{i=1}^k p_ix_i \biggr)\biggl(x_i-\sum_{i=1}^kp_ix_i\biggr)^\tp \in \mathbb{R}^{n\times n}.
\]
Let the eigenvalue decomposition of the symmetric positive semidefinite matrix $X$ be $X=Q\Lambda Q^\tp $ with $\Lambda = \diag(\lambda_1,\dots,\lambda_n)$, $\lambda_1\ge \dots \ge \lambda_n\ge 0$. Then 
\[
\inf_{V\in \O(m,n)}\tr(VXV^\tp )=\sum_{i=0}^{m-1}\lambda_{n-i}
\]
and is attained when $V \in \O(m,n)$ has row vectors given by the last $m$ columns of $Q \in \O(n)$.

\end{example}

\begin{example}[$2$-Wasserstein distance between discrete measures on $\mathbb{R}^m$ and  $\mathbb{R}^n$]
More generally, we may seek the $2$-Wasserstein distance between  discrete probability measures $\rho_1 \in \M(\mathbb{R}^m)$ and $\rho_2 \in \M(\mathbb{R}^n)$. Let $\rho_1$ be supported on $x_1,\dots,x_k \in \mathbb{R}^m$ with values $\rho_1(x_i) = p_i$, $i = 1,\dots, k$; and $y_1,\dots,y_l \in \mathbb{R}^n$ with values $\rho_2(x_i) = q_i$, $i = 1,\dots, l$. The optimization problem for $\PW_2(\rho_1, \rho_2 )$ becomes
\begin{equation}\label{eq:disc}
\inf_{V \in \O(m,n),\;b\in\mathbb{R}^m,\; \pi \in \Gamma(\rho_1 ,\rho_2 )} \sum_{i=1}^k\sum_{j=1}^l \pi_{ij}\|Vx_i+b-y_j\|_2^2,
\end{equation}
where
\begin{align*}
\Gamma(\rho_1 ,\rho_2 ) = \biggl\{ \pi \in \mathbb{R}^{k \times l}_\p : \sum_{j=1}^l \pi_{ij} &= p_i,\; i =1,\dots,k; \\
\sum_{i=1}^k \pi_{ij} &= q_j, \; j = 1,\dots,l \biggr\}.
\end{align*}
While the solution to \eqref{eq:disc} may no longer be determined in closed-form, it is a polynomial optimization problem and can be solved using the Lasserre sum-of-squares technique as a sequence of semidefinite programs \cite{lasserre2015introduction}.
\end{example}

\section{Conclusion}

We proposed a simple, natural framework for taking any $p$-Wasserstein metric or $f$-divergence, and constructing a corresponding distance for probability distributions on $m$- and $n$-dimensional measure spaces where $m \ne n$. The new distances preserve some well-known properties satisfied by the original distances. We saw from several examples that the new distances may be either determined in closed-form or near closed-form, or computed using Stiefel manifold optimization or sums-of-squares polynomial optimization. In future work, we hope to apply our framework to other distances like the Bhattacharyya distance \cite{Bhattacharyya}, the L\'evy--Prokhorov metric \cite{Levy,Prokhorov}, and the {\L}ukaszyk--Karmowski metric \cite{Lukaszyk}.

\subsection*{Acknowledgment} We are deeply grateful to the two anonymous referees for their exceptionally helpful comments and suggestions that vastly improved our article. This work is supported by DARPA HR00112190040, NSF IIS 1546413 and DMS 1854831, and the Eckhardt Faculty Fund. YC would like to express his gratitude to Philippe Rigollet, Jonathan Niles-Weed, and Geoffrey Schiebinger for teaching him about optimal transport and for their hospitality when he visited MIT in 2018. LHL would like to thank Louis H.~Y.~Chen for asking the question on p.~\pageref{question}.

\bibliographystyle{IEEEtran}

\end{document}